\documentclass[a4paper, 12pt, oneside]{amsart}

\usepackage[utf8]{inputenc}
\usepackage[T1]{fontenc}
\usepackage[english]{babel}

\usepackage[a4paper, top=2.5cm, bottom=2.5cm, left=2.5cm, right=2.5cm]{geometry} % margini più stretti
\usepackage{amsmath}
\usepackage{amssymb}
\usepackage{amsthm}
\usepackage{latexsym}
\usepackage{mathtools}

\usepackage{graphicx} % Required for inserting images
\usepackage{pdfpages}
\usepackage{enumitem}
\usepackage{tikz-cd}
\usepackage[all,cmtip]{xy}

\usepackage[colorlinks=true,
linkcolor=blue, citecolor=green]{hyperref}

\theoremstyle{plain}% default
\newtheorem{thm}{Theorem}[section]
\newtheorem{lem}[thm]{Lemma}

\newtheorem{cor}[thm]{Corollary}
\newcommand{\mc}[1]{\mathcal{#1}}

\theoremstyle{remark}

\theoremstyle{definition}
\newtheorem{definition}[thm]{Definition}
\newtheorem{oss}[thm]{Remark}

\newtheorem*{notation}{Notation}

\numberwithin{equation}{section}

\title{On the correspondence between Ulrich bundles and curves on surfaces}

\author{Sofia Bordoni*}
\thanks{* Ph.D. student in Mathematics, Sapienza University of Rome, Rome, Italy. \\ Email: sofia.bordoni@uniroma1.it}
\begin{document}
	\maketitle
	\begin{abstract}
		This work provides a curve-based approach to Ulrich bundles on surfaces, establishing a correspondence that characterizes their existence, with a focus on applications to surfaces in $\mathbb{P}^3$. 
	\end{abstract}
	\section{Introduction}
	Ulrich bundles, first introduced in commutative algebra in the 1980s \cite{ulrich1984gorenstein, brennan1987maximally}, also have a natural geometric counterpart: they arise in connection with the \emph{determinantal representation problem}, concerning the question of whether the defining equation of a hypersurface can be expressed as the determinant of a matrix of linear forms (see \cite{beauville2018introduction} for further details).
	
	Many constructions and examples have appeared in the literature (see, e.g., \cite{eisenbud2003resultants, beauville2018introduction}); nevertheless, the general \textit{existence problem} remains largely open, even for low-dimensional varieties such as surfaces. In the case of surfaces, some useful results have already been presented: for instance, in \cite{casnati2019ulrich}, Ulrich bundles were characterized in terms of their Chern classes and only two cohomological vanishings, providing a structural understanding in dimension two.
	
	The aim of this article is to advance this line of research by developing \textit{a curve-based approach} to Ulrich bundles on smooth projective surfaces, leading to new existence results and geometric applications. Specifically, it is proved that an Ulrich bundle on a surface exists if and only if it arises as a \textit{Lazarsfeld-Mukai bundle} associated with a triple $(C,W,\mathcal{L})$, where $C$ is a smooth curve lying on the surface, and $(W, \mathcal{L})$ is a linear series on $C$, satisfying some explicit conditions. The precise statement is given in the following theorem. 
	\begin{thm} \label{main}
		Let \( S \subset \mathbb{P}^N \) be a smooth projective surface of degree $d\geq 2$, embedded by the linear system $|H|$, where $H\in |\mc{O}_S(1)|$. Then, there exists an Ulrich bundle $\mc{E}$ of rank $r$ on $S$ if and only if there exists a smooth (possibly disconnected) curve $C \subset S $ of genus $g$ together with a pair $(W,\mc{L})$, where $\mc{L}$ is a line bundle on $C$ and $W\subseteq H^0(\mc{L})$ is a $r-$dimensional base-point free linear series, such that: 
		
		\begin{enumerate}[label=(\roman*)]
			\item $H^1(C, \mc{L}(K_{S}+H))=0$; %where $\mc{L}=\mc{O}_C(D)$ for some divisor D on $\mc{C}$;
			\item the multiplication map 
			\begin{center}
				$\varphi: W\otimes H^0(S,\mc{O}_{S}(K_S+2H)) \to H^0(\mc{L}(K_S+2H))$
			\end{center}
			is injective;
			\item $ \deg(C)= \frac{r}{2}(K_{S}+3H)\cdot H$ and \\[1ex] $\deg(\mc{L})= r\chi(\mc{O}_{S})+g-1-C\cdot K_S-rH^2$.
			\end{enumerate} 
	\end{thm}
%	For the detailed proof, see \hyperlink{dimmain}{\textit{Proof of Theorem \ref{main}}}.
	
%	\textit{Outline of the proof: } Starting from a smooth curve $C \subset S$ together with a line bundle $\mathcal{L}$ on $C$ and an $r$-dimensional base-point-free subspace $W \subset H^0(\mathcal{L})$ satisfying (i), (ii), and (iii), the Lazarsfeld-Mukai bundle associated with the triple $(C, W, \mc{L})$ is a rank $r$ Ulrich bundle on $S$. Conversely, given a rank $r$ Ulrich bundle $\mathcal{E}$ on $S$, one considers a general $r$-dimensional subspace $V \subset H^0(\mathcal{E})$ and the associated evaluation map $\varphi_V: V \otimes \mathcal{O}_S \to \mathcal{E}$. The $(r-1)$-st degeneracy locus of $\varphi_V$ defines a smooth curve $C \subset S$, the line bundle $\mathcal{L}$ is obtained as $\mathcal{E}xt^1(\mathrm{Coker}(\varphi_V), \mathcal{O}_S)$, and the subspace $W \subset H^0(\mathcal{L})$ is identified with $V^*$, providing the triple $(C, W, \mathcal{L})$ associated to $\mathcal{E}$; this triple then satisfies conditions (i), (ii), and (iii). For the detailed proof, see

	\vspace{1ex}
	
	Once the theorem is settled, one may consider the curve’s geometry. The next corollary gives a precise answer concerning the connectedness of the curve $C$ of Theorem \ref{main}.

	\begin{cor}\label{connessione}
		Let \( S \subset \mathbb{P}^N \) be a smooth projective surface of degree $d\geq 2$, embedded by the linear system $|H|$, where $H\in |\mc{O}_S(1)|$. Let $\mc{E}$ be the Ulrich bundle on $S$ corresponding to the triple $(C, W, \mc{L})$, as in Theorem \ref{main}. Then $C$ is irreducible if and only if one of the following cases arises: 
		\begin{enumerate}[label=(\roman*)]
			\item $(S, \mc{O}_S(1), \mc{E})\ncong (\mathbb{P}\mc{F}, \mc{O}_{\mathbb{P}\mc{F}}(1), \pi^*(\mc{G}(\det (\mc{F})))$ where \( \mc{F} \) is a rank 2 very ample vector bundle over a smooth curve \( B \) of genus g, \( \mc{G} \) is a rank \( r \) vector bundle on \( B \) such that \( H^q(\mc{G}) = 0 \) for \( q \geq 0 \); 
			%and the triple $(r,d,g)\neq (1,2,0)$;
			\item $(S, \mc{O}_S(1), \mc{E})\cong (\mathbb{P}^1\times \mathbb{P}^1, \mc{O}_{\mathbb{P}^1\times \mathbb{P}^1}(1), \mc{O}_{\mathbb{P}^1\times \mathbb{P}^1}(1,0))$ or 
			\newline $(S, \mc{O}_S(1), \mc{E})\cong (\mathbb{P}^1\times \mathbb{P}^1, \mc{O}_{\mathbb{P}^1\times \mathbb{P}^1}(1), \mc{O}_{\mathbb{P}^1\times \mathbb{P}^1}(0,1))$.
		\end{enumerate}
	\end{cor}
%	For the complete proof, see \hyperlink{dimconnessione}{\textit{Proof of Corollary \ref{connessione}}}.
	
	\vspace{1ex}
	
	Once the correspondence is established, one can further investigate the geometry of the curve $C$.
	For example, when we work inside $\mathbb{P}^{3}$, the connectedness criterion collapses to a single exceptional configuration: except for the ruled surface $\mathbb{P}^{1}\times\mathbb{P}^{1}$ carrying spinorial Ulrich bundles, the associated curve is automatically connected (Corollary \ref{connessione P^3}).
	
	Moreover, the correspondence provides quantitative information. In particular, one obtains \textit{bounds for the genus} of the curve $C$: 
	an upper bound via Hodge index theorem, and a lower bound via Bogomolov’s inequality (Remark \ref{bounds}). 
	These estimates become especially sharp for \textit{quartic surfaces in $\mathbb{P}^3$}, where the lower bound can be refined, providing a nontrivial improvement over the classical inequalities (\hyperlink{refining}{Section 7}).
	
	Finally, the correspondence has applications to the study of \textit{Noether–Lefschetz loci}: it is proved that the subset of surfaces supporting an Ulrich line bundle coincides with an entire irreducible component of the Noether-Lefschetz locus (Theorem \ref{componente NL}). 
	%	The insights derived from our main theorem are then applied to specific cases, particularly to \textit{surfaces in $\mathbb{P}^3$} offering new results about Ulrich bundles on them and their relation with \textit{Noether-Lefschetz loci}. 
	%	
	%	In particular, when we work inside $\mathbb{P}^{3}$, the connectedness criterion collapses to a single exceptional configuration: except for the ruled surface $\mathbb{P}^{1}\times\mathbb{P}^{1}$ carrying spinorial Ulrich bundles, the associated curve is automatically connected.  
	\vspace{1ex}
	\paragraph{\textbf{Acknowledgements}} The author is deeply grateful to Prof. Angelo Felice Lopez for his guidance and helpful advice. This work derives from the author’s Master’s thesis, completed in July 2025 at Roma Tre University under his supervision.
	\section{Preliminaries and Notation}
	This section introduces the fundamental concepts of Ulrich and Lazarsfeld-Mukai bundles, presenting their definitions and essential properties, and establishing the notation that will be used throughout the paper.
	\begin{definition}[\textbf{Ulrich bundle}] \label{def Ulrich}
		Let $X\subseteq \mathbb{P}^N$ be a smooth projective variety and let $\mc{E}$ be a vector bundle on $X$. $\mc{E}$ is said to be \emph{Ulrich} if it satisfies
		$H^i(X, \mc{E}(-p))=0$ for all $i\geq 0$ and $1\leq p \leq \dim X$.
	\end{definition}
	
	\begin{definition}[\textbf{Ulrich complexity}]
		Let \(X\subseteq \mathbb P^{N}\) be a smooth variety.
		The \emph{Ulrich complexity} of \(X\) is the integer
		\[uc(X):=
		\min\{r\ge 1 \  : \  \text{\(X\) carries an Ulrich vector bundle of rank \(r\)}\}.
		\]
		If no Ulrich vector bundle exists on \(X\), we set
		\(uc(X):=\infty\).
	\end{definition}
	
	\begin{oss} \label{proprietà Ulrich}
		Let $X\subseteq \mathbb{P}^N$ be a smooth projective variety of dimension $n$ and degree $d$ and let $\mc{E}$ be a rank $r$ Ulrich vector bundle on $X$. Then $\mc{E}$ is 0-regular, hence it is globally generated by \cite[Thm. 1.8.5]{lazarsfeld2017positivity}. Moreover, $h^0(\mc{E})=rd$ and $\chi(\mc{E}(m))=rd\binom{m+n}{n}$ (see \cite[(3.1)]{beauville2018introduction} and \cite[Lemma 2.6]{casanellas2012stable} ).
	\end{oss}
	
	\begin{oss}[\textbf{Stability of Ulrich bundles}] \label{stability}
		According to \cite[Thm. 2.9]{casanellas2012stable}, Ulrich bundles on smooth projective varieties are \textit{semi-stable}. Minimal-rank Ulrich bundles -- those whose rank coincides with the Ulrich complexity of the variety -- are even \textit{stable}, by \cite[Thm. 2.9]{casanellas2012stable} via an argument analougous to that in \cite[Proof of Lemma 2.3]{casnati2018stability}. Consequently, by \cite[Cor. 1.2.8]{huybrechts2010geometry}, a minimal-rank Ulrich bundle $\mc{E}$ on a smooth projective variety $X$ is also \textit{simple}, i.e., it satisfies
		\(
		h^0(X, \mc{E} \otimes \mc{E}^*) = 1.
		\)
		
	\end{oss}
	
	\begin{definition}[\textbf{Lazarsfeld-Mukai bundles}]
		Let \( S \subseteq \mathbb{P}^N \) be a smooth projective surface. Given a smooth curve $C \subset S$, a line bundle $\mc{L}$ on $C$, and a base-point-free subspace $W \subset H^0(\mc{L})$ of dimension $r$, the \emph{kernel bundle} $K_{C,W,\mc{L}}$ is defined as the kernel of the evaluation map
		\(
		W \otimes \mathcal{O}_S \twoheadrightarrow \mc{L}.
		\)
		
		Its dual $K_{C,W,\mc{L}}^*$ is called the \emph{Lazarsfeld--Mukai bundle} associated with $(C,W,\mc{L})$.
		
	\end{definition}

	\begin{oss}\label{LazMu}
		The bundle $K_{C,W,\mc{L}}$ is a vector bundle on $S$ of rank $r$, with first Chern class $c_1(K_{C,W,\mc{L}}) = -[C]$ and second Chern class $c_2(K_{C,W,\mc{L}}) = \deg \mc{L}$ (see \cite[\S 3.3]{lazarsfeld1989sampling}).
		
	\end{oss}
	The following lemma already highlights the central role of Lazarsfeld--Mukai bundles in the correspondence of Theorem \ref{main}. In particular, the reader should observe that, together with \cite[Prop. 2.1]{casnati2019ulrich}, it immediately follows that the Lazarsfeld--Mukai bundle associated with a triple $(C,W,\mathcal{L})$ satisfying properties (i), (ii), and (iii) of Theorem \ref{main} is indeed an Ulrich bundle.
	
	\begin{lem}\label{main lemma}
		Let \( S \subseteq \mathbb{P}^N \) be a smooth projective surface of degree $d$, embedded by the linear system $|H|$, where $H\in|\mc{O}_S(1)|$, and let $C \subset S $ be a smooth curve. Let $\mc{L}$ be a line bundle on $C$, $W\subseteq H^0(\mc{L})$ a base-point-free linear series of dimension $r$ and let \( \mathcal{E} \) be the Lazarsfeld–Mukai bundle on $S$ associated to $(C,W, \mc{L})$, i.e., $\mc{E}^*=K_{C,W,\mc{L}}.$ 
		%such that the following exact sequence holds: 
		%\[  0\to \mc{E}^*\to W\otimes \mc{O}_S\xrightarrow{\varphi} \mc{L}\to 0 \qquad (*)\]
		%where $\varphi$ is the canonical evaluation map. 
		Then
		\begin{enumerate}[label=(\roman*)]
			\item $h^0(\mc{E}(-1)) = h^1(\mc{L}(K_S+H))$
			\item $H^2(\mc{E}(-2))= \ker \varphi$, where $\varphi$ is the multiplication map,
			\[ \varphi: W\otimes H^0(S,\mc{O}_{S}(K_S+2H)) \to H^0(\mc{L}(K_S+2H)).\]
			\item  $\begin{cases}
				\chi(\mc{E}(-1))=0 \\[1ex]
				\chi(\mc{E}(-2))=0
			\end{cases}
			\iff \begin{cases}
				\deg (C)=\frac{r}{2}(K_S+3H)\cdot H \\[1ex]
				\deg(\mc{L})= r\chi(\mc{O}_{S})+g-1-C\cdot K_S-rH^2
			\end{cases}$
			
			$\\[0.5ex]$ where $g$ is the genus of the curve $C$. 
		\end{enumerate}
	\end{lem}
	\begin{proof}
		\textit{(i)} Consider the exact sequence \begin{equation} \label{se1}
			0\to \mc{E}^* \to W\otimes \mc{O}_S\to \mc{L}\to 0. 
		\end{equation}
		By tensoring (\ref{se1}) with $\mc{O}_S(K_S+H)$ and by taking cohomology, Kodaira's vanishing implies that $H^1(\mc{O}_S(K_S+H))=H^2(\mc{O}_S(K_S+H))=0$, and hence $H^1(\mc{L}(K_S+H))\cong H^2(\mc{E}^*(K_S+H))$. The statement then follows from Serre's duality.
		
		\textit{(ii)} Tensoring the exact sequence (\ref{se1}) by $\mc{O}_S(K_S+2H)$ and passing to cohomology identifies $\ker\varphi$ with $H^0(\mc{E}^*(K_S+2H))$; whence (ii) follows by Serre's duality.
		
		\textit{(iii)} A straightforward Riemann-Roch computation together with the standard equalities for the first and second Chern classes of $\mc{E}(-1)$ and $\mc{E}(-2)$ (see \cite[Prop. 5.17]{eisenbud20163264}) yields the equivalence  
		\begin{equation} \label{eq casnati}
			\begin{cases}
				\chi(\mc{E}(-1))=0 \\[1ex]
				\chi(\mc{E}(-2))=0
			\end{cases}
			\iff
			\begin{cases}
				c_1(\mc{E})\cdot H=\frac{r}{2}(K_S+3H)\cdot H \\[1ex]
				c_2(\mc{E})=r\chi(\mc{O}_S)+\frac{1}{2}c_1(S)c_1(\mc{E})+\frac{1}{2}c_1(\mc{E})^2-rH^2.
			\end{cases}
		\end{equation}
		
		Statement (iii) then follows immediately from the equivalence (\ref{eq casnati}) together with the well-known identities for the Chern classes of the Lazarsfeld-Mukai bundle associated with the triple $(C, W, \mc{L})$ established in Remark \ref{LazMu}. 
	\end{proof}
	
	\section{Proof of Theorem \ref{main}}
	
	\hypertarget{dimmain}{}
	\begin{proof}[Proof of Theorem \ref{main}] 
		Let us assume that $\mc{E}$ is a rank $r$ Ulrich vector bundle on $S$. Observe that $\mc{E}$ is globally generated, therefore, there exists a general subspace $V$ of $H^0(\mc{E})$ of dimension $r$. Hence, the evaluation map,  which is injective, 
		\begin{center}
			$\varphi_{V}: V\otimes \mc{O_{S}}\to \mc{E}$
		\end{center}
		
		is general in $\mc{H}om(V\otimes \mc{O}_S, \mc{E})$. 
		
		In addition, observe that $\mc{H}om(V\otimes \mc{O}_S, \mc{E})\cong \mc{O}_{S}^{\oplus r} \otimes \mc{E}$ is globally generated, since it is the tensor product of two globally generated vector bundles. We now consider the $(r-1)$-st degeneracy locus of $\varphi_{V}$, $D_{r-1}(\varphi_{V})$. According to \cite[Thm. 2.8]{ottaviani1995varieta}, this locus is either empty or has codimension 1.

		We now claim that if $D_{r-1}(\varphi_{V})=\emptyset $ then $S$ must be a surface of degree $d=1$, which is excluded by the assumptions. In fact, if $D_{r-1}(\varphi_{V})=\emptyset$, then $rk(\varphi(x))=r$ $\forall x\in S$, so that $\varphi_{V}$ is surjective at every point of $S$. It follows that $\mc{E}\cong \mc{O}_{S}^{\oplus r}$; in particular, since $\mc{O}_{S}^{\oplus r}$ is Ulrich, so is $O_S$  hence, by Remark \ref{proprietà Ulrich}, one obtains $d=h^0(\mc{O}_S)=1$. 
		
		Hence, the case $D_{r-1}(\varphi_V)=\emptyset$ cannot occur, and, therefore, $D_{r-1}(\varphi_V)$ has the expected codimension one, so that it defines a curve on $S$. Let 
		\[C:= D_{r-1}(\varphi_{V}).\]
		By \cite[Thm. 2.8]{ottaviani1995varieta}, Sing($C$)=$D_{r-2}(\varphi_{V})$. Moreover, according to the same theorem, if $D_{r-2}(\varphi_{V})$ were non empty, it would have codimension $4$ which is clearly a contraddiction. Hence, Sing($C$)$=\emptyset$, implying that $C$ is a smooth curve on $S$.
		
		\vspace*{1ex}
		We now proceed to define an appropriate line budle $\mc{L}$ on $C$, as well as a suitable $r$-dimensional base-point free subspace $W\subseteq H^0(\mc{L)}$, which we will later prove to satisfy properties (i), (ii), and (iii).
		
		Consider the exact sequence
		\begin{equation}\label{1}
			0\to \mc{O}_S^{\oplus r}\xrightarrow{\varphi_V} \mc{E}\to \mc{L'}:=\text{Coker}(\varphi_V)\to 0. 
		\end{equation}
		By construction, the cokernel $\mc{L'}$ is supported on $C = D_{r-1}(\varphi_V)$. As $D_{r-2}(\varphi_V) = \emptyset$, the restriction of $\mc{L'}$ to $C$ has constant rank one, and therefore $\mc{L'}$ is a line bundle on $C$.
		
		Dualizing sequence (\ref{1}), and noting that $(\mc{L'})^* = 0$ since $\mc{L'}$ is torsion on $S$, we obtain
		\[0\to \mc{E}^* \to V^*\otimes \mc{O}_S\to \mc{L}\to 0,\]
		where $\mc{L} := \mc{E}xt^1(\mc{L'}, \mc{O}_S)$. Finally, one has $\mc{L}\cong \mc{N}_{C/S}\otimes \mc{L'}^*$ (see  \cite[Proof of Prop. 1.1]{green1987special}) so that $\mc{L}$ is a line bundle on $C$.
		
		The above exact sequence induces the long exact sequence in cohomology
		\[0\to H^0(\mc{E^*})\to V^*\otimes H^0( \mc{O}_S)\xrightarrow{\psi} H^0(\mc{L})\to H^1(\mc{E}^*)\to \ldots\]
		By \cite[Lemma 2.1]{ciliberto2023ulrich}, $H^0(\mc{E}^*) = 0$, so the map $\psi: V^* \to H^0(\mc{L})$ is injective. Identifying $V^* \otimes H^0(\mc{O}_S) \cong V^*$, we set
		\[
		W := \mathrm{Im}(\psi) \subset H^0(\mc{L}),
		\]
		which gives $W \cong V^*$ and $\dim(W) = r$. For each point $p \in C$, the fiber map $\psi_p: V^* \to \mc{L}_p$ is nonzero, hence $W$ is base-point free on $C$.
		
		The triple $(C, W, \mc{L})$ thus satisfies the hypotheses of Lemma \ref{main lemma}: $C \subset S$ is smooth, $\mc{L}$ is a line bundle on $C$, $W \subset H^0(\mc{L})$ is an $r$-dimensional base-point-free linear series, and $\mc{E}^* \cong K_{C,W,\mc{L}}$. Since $\mc{E}$ is Ulrich, we have
		\[
		H^0(\mc{E}(-1)) = 0, \quad H^2(\mc{E}(-2)) = 0, \quad \text{and} \quad \chi(\mc{E}(-1)) = \chi(\mc{E}(-2)) = 0.
		\]
		A direct application of Lemma \ref{main lemma} then shows that the triple $(C, W, \mc{L})$ satisfies conditions (i), (ii), and (iii), completing the first part of the proof. 
		
		Conversely, suppose one is given a smooth curve $C\subset S$, a line bundle $\mc{L}$ on $C$, and an 
		$r-$dimensional, base‑point‑free subspace $W\subseteq H^0(\mc{L})$ satisfying conditions (i), (ii) and (iii) and set \( \mc{E}:=K_{C,W,\mc{L}}^*.\)
		
		Observe that $\mc{E}$ is a vector bundle of rank $r$ on $S$ (see Renark \ref{LazMu}). 
		
		By Lemma \ref{main lemma}(i) and (ii), conditions (i) and (ii) translate into 
		\[ H^0(\mc{E}(-1))=0 \qquad \text{and} \qquad H^2(\mc{E}(-2))=0. \] 
		Moreover, Lemma \ref{main lemma}(iii), together with equivalence (\ref{eq casnati}), shows that the numerical conditions in (iii) are equivalent to the Chern–class identities 
		\[
		c_1(\mc{E})\cdot H=\tfrac{r}{2}(K_S+3H)\cdot H, 
		\qquad 
		c_2(\mc{E})=r\chi(\mc{O}_S)+\tfrac{1}{2}c_1(S)c_1(\mc{E})+\tfrac{1}{2}c_1(\mc{E})^2-rH^2.
		\]
		Hence, $\mc{E}$ is an Ulrich bundle of rank $r$ on $S$ by \cite[Prop. 2.1]{casnati2019ulrich}.
	\end{proof}
	We record some supplementary points regarding the multiplication map from point (ii) of Theorem \ref{main}. 
	\begin{oss}
		Within the framework of Theorem \ref{main}, let $\mc{E}$ be be an Ulrich bundle on $S$ with associated triple $(C,W, \mc{L})$, or conversely let $(C,W,\mc{L})$ be any triple satisfying (i), (ii) and (iii) with associated Ulrich bundle $\mc{E}$. Then the multiplication map
		\[\varphi: W\otimes H^0(S,\mc{O}_{S}(K_S+2H)) \to H^0(\mc{L}(K_S+2H))\] 
		is not merely injective but in fact an isomorphism.
	\end{oss}
	\begin{proof}
		Observe that it suffices to show that \(
		rh^0(\mathcal{O}_S(K_S + 2H)) = h^0(\mathcal{L}(K_S + 2H)),
		\) 	where we set $r := \dim W$. In both cases, we have $\mathcal{E}^* \cong K_{C,W,\mathcal{L}}$. Hence, by Serre's duality and the fact that $\mathcal{E}$ is Ulrich, we obtain
		\(
		\chi(\mathcal{E}^*(K_S+2H)) = \chi(\mathcal{E}(-2H)) = 0.
		\)
		Consider the exact sequence
		\[
		0 \to \mathcal{E}^* \to W \otimes \mathcal{O}_S \to \mathcal{L} \to 0
		\]
		and tensor it with $\mathcal{O}_S(K_S + 2H)$. By additivity of the Euler characteristic and Kodaira's vanishing, we obtain
		\(
		rh^0(\mathcal{O}_S(K_S + 2H)) = \chi(\mathcal{L}(K_S + 2H)). 
		\)
		From the short exact sequence
		\[0\to \mc{L}(K_S+H)\to \mc{L}(K_S+2H)\to \mc{L}(K_S+2H)_{|H_C}\to 0\]
		together with $H^1(\mathcal{L}(K_S+H)) = 0$ (by condition (i) of Theorem \ref{main}) and the vanishing of $H^1(\mathcal{L}(K_S+2H)|_{H_C})$ (since $\mc{L}(K_S+2H)_{|H_C}$ is supported on finitely many points), we deduce $H^1(\mathcal{L}(K_S+2H)) = 0$. Hence,
		\(
		\chi(\mathcal{L}(K_S+2H)) = h^0(\mathcal{L}(K_S+2H))
		\), thereby implying \(
		rh^0(\mathcal{O}_S(K_S + 2H)) = h^0(\mathcal{L}(K_S + 2H))
		\) and thus concluding the proof.
		\end{proof}
		
		\begin{oss}
			Recall that the multiplication map $\varphi$ appearing in point (ii) of Theorem \ref{main} is defined by restriction to $C$; in other words, $\varphi=\psi \circ \rho$ where \[\psi: W\otimes H^0(\mc{O}_S(K_S+2H)) \to W\otimes H^0(\mc{O}_S(K_S+2H)_{|C})\] and $\rho$ is the natural map $W\otimes H^0(\mc{O}_S(K_S+2H)_{|C})\to H^0(\mc{L}(K_S+H))$.
			
			Under the hypotheses of Theorem \ref{main} (in particular $d\ge2$) and assuming $r\ge2$, the map $\psi$ is always injective.
		\end{oss}
		\begin{proof}
			Since 
			\(
			\ker\psi \cong W\otimes H^0\bigl(\mathcal{O}_S(K_S+2H - C)\bigr),
			\)
			it suffices to show 
			\(
			(K_S+2H - C)\cdot H <0,
			\)
			so that $\mathcal{O}_S(K_S+2H - C)$ has no nonzero global sections.
			
			Using $\deg(C)=C\cdot H = \tfrac r2(K_S+3H)\cdot H$ and the adjunction formula $K_S\cdot H + H^2 = 2g(H)-2$, we find
			\[(K_S+2H - C)\cdot H= -\tfrac12\Bigl((r-2)(K_S\cdot H + 2d) +r\,d\Bigr).\]
			Since $r\ge2$ and $d\ge2$, the right-hand side is negative, hence $\psi$ is injective.
			\end{proof}
	
	\section{Curves of Theorem \ref{main}: Connectivity and Genus}
	
	The next remark provides bounds on the genus of the curves in Theorem \ref{main}, with the Hodge inequality giving an upper bound and Bogomolov’s inequality a lower bound.
	
	\begin{oss}\label{bounds}
		Let \( S \subset \mathbb{P}^N \) be a smooth projective surface of degree $d\geq 2$, embedded by the linear system $|H|$, where $H\in |\mc{O}_S(1)|$. In the setting of Theorem \ref{main}, let $\mc{E}$ be be an Ulrich bundle on $S$ with associated triple $(C,W, \mc{L})$, or conversely let $(C,W,\mc{L})$ be any triple satisfying (i), (ii) and (iii) with associated Ulrich bundle $\mc{E}$. Let $g$ be the genus of $C$. Then 
		\begin{enumerate} [label=(\roman*)]
			\item $g\leq \frac{r^2}{8d}K_S\cdot H(K_S\cdot H+6d)+ \frac{1}{2}C\cdot K_S+ \frac{9r^2d}{8}+1$;
			\item $g\geq \frac{r+1}{2}C\cdot K_S+ r^2(d-\chi(\mc{O}_S))+1.$
		\end{enumerate}
	\end{oss}
	
	\begin{proof}
		(i) It follows directly from \cite[Thm. V.1.9]{hartshorne2013algebraic}, observing that $C^2=2g-2-C\cdot K_S$ by adjunction and $C\cdot H= \frac{r}{2}(K_{S}+3H)\cdot H $ by condition (iii) of Theorem \ref{main}.
		
		(ii) $\mc{E}$ is an Ulrich vector bundle, hence, it is semi-stable. Then, by \cite[Thm. 0.3]{gieseker1979theorem}, one has
		\[2rc_2(\mc{E}) - (r-1)c_1(\mc{E})^2 \geq 0,\]
		that is, observing that $c_1(\mc{E}) = C$ and that $c_2(\mc{E}) = \deg \mc{L}= r\chi(\mc{O}_S) + \frac{1}{2}C^2 - \frac{1}{2}C \cdot K_S - rd$ by Remark \ref{LazMu} and condition (iii) of Theorem \ref{main} together with the adjunction,
		\begin{equation}\label{bog}
			2r^2\chi(\mc{O}_S)+C^2-rC\cdot K_S-2r^2d\geq 0
		\end{equation}
		Then (ii) follows from (\ref{bog}) using the identity $C^2=2g-2-C\cdot K_S$.
	\end{proof}
	
	Our next goal is to investigate the connectedness of the curve. To this end, we begin with the following observation, which is a corollary of \cite[Thm. 1]{lopez2021classification}.
	\begin{oss}\label{LM}
		Let $S\subseteq \mathbb{P}^N$ be a smooth projective surface, and let $H$ be a very ample line bundle on $S$. Suppose that $\mc{E}$ is an Ulrich bundle with respect to $(S,H)$. Then the following conditions are equivalent:
		\begin{enumerate}[label=(\roman*)]
			\item $(S, \mc{O}_S(1), \mc{E})$ is isomorphic either to $(\mathbb{P}^2, \mathcal{O}_{\mathbb{P}^2}(1), \mathcal{O}_{\mathbb{P}^2}^{\oplus r})$ or to \newline $(\mathbb{P}\mc{F}, \mc{O}_{\mathbb{P}\mc{F}}(1), \pi^*(\mc{G}(\det (\mc{F})))$ where \( \mc{F} \) is a rank 2 very ample vector bundle over a smooth curve \( B \) and \( \mc{G} \) is a rank \( r \) vector bundle on \( B \) such that \( H^q(\mc{G}) = 0 \) for \( q \geq 0 \);
			\item $c_1(\mc{E})^2=0$;
			\item the bundle $\mc{E}$ is not big.
		\end{enumerate}
	\end{oss}
	\begin{proof}
		Assume (i) holds. If  $(S, \mc{O}_S(1), \mc{E})\cong (\mathbb{P}^2, \mathcal{O}_{\mathbb{P}^2}(1), \mathcal{O}_{\mathbb{P}^2}^{\oplus r}) $, then $c_1(\mc{E})^2=0$. Otherwise, set $\mc{M}:=\mc{G}\otimes \det(\mc{F})$, so that $\mc{E}=\pi ^*(\mc{M})$, and $\mc{L}:= \det (\mc{M})$. Since $\pi : S\to B$ is a ruled surface on a curve, we have
		\[c_1(\mc{E})^2= (\det(\mc{E}))^2= (\det (\pi^*\mc{M}))^2= (\pi^*(\det \mc{M}))^2= (\pi^* \mc{L})^2=\pi^*(\mc{L}^2)=0\]
		because the self-intersection number of a line bundle on a curve is always zero. Hence, condition (ii) follows.
		
		Now assume (ii) holds. According to \cite[Remark 2.2]{lopez2021classification}, if an Ulrich bundle $\mc{E}$ satisfies $c_1(\mc{E})^2=0$ then $\mc{E}$ cannot be big, thus condition (iii) follows. 
		
		Finally, if (iii) holds, then \cite[Thm. 1]{lopez2021classification} asserts that $\mc{E}$ must be of the specific form described in (i). 
		
		Therefore, the three statements are equivalent.
	\end{proof}
	
	We proceed by presenting the following corollary of Theorem \ref{main}. 
	\begin{cor} \label{prop c^2=0}
		Let \( S = \mathbb{P}\mathcal{F} \subset \mathbb{P}^N \) be a ruled surface of degree \( d\geq 2 \) over a smooth curve $B$, with $H\in |\mc{O}_{\mathbb{P}\mc{F}}(1)|$.  
		Then there exists a rank \( r \) Ulrich bundle \( \mathcal{E} \) on \( S \) with \( c_1(\mathcal{E})^2 = 0 \)  
		if and only if \( S \) contains a smooth curve \( C \), which is the disjoint union of  
		\[
		t = r\,(g(B) - 1 + d)
		\]  
		fibers (lines) of $S$, together with a base point free \( r \)-dimensional subspace \( W \subseteq H^0(\mathcal{O}_C) \) such that the multiplication map \(W\otimes H^0(\mc{O}_S(K_S+2H))\to H^0(\mc{O}_C)\)
		is injective. 
	\end{cor}
	Before proving Corollary \ref{prop c^2=0}, we record a brief technical remark.
	\begin{oss} \label{remark tecnico}
		Let \( S \cong \mathbb{P}\mathcal{F} \subset \mathbb{P}^N \) be a ruled surface of degree \( d\geq 2 \) over a smooth curve $B$, with $H\in |\mc{O}_{\mathbb{P}\mc{F}}(1)|$. Let $C$ be disjoint union of $t$ fibers of $S$. Then 
		\begin{enumerate} [label=(\roman*)]
			\item $r\chi(\mc{O}_{S})+g(C)-1-C\cdot K_S-rH^2=t-r(g(B)-1+d)$
			\item Let $\mc{L}$ be a line bundle on $C$. Then $\mc{L}\cong \bigoplus_{i=1}^t \mc{O}_{\mathbb{P}^1}(a_i)$ for some integers $a_i$. 
			
			In particular, $H^1(\mc{L}(K_S+H))=0$ if and only if $a_i\geq 0$ for every $i=1,\ldots, t$
			%\item $H^0(\mc{O}_S(K_S+2H-C))=0$
		\end{enumerate}
	\end{oss}
	\begin{proof}
		(i) By adjunction, one has
		\[ r\chi(\mathcal{O}_S) + g(C) - 1 - C \cdot K_S - rH^2
		= r\chi(\mathcal{O}_S) + \frac{1}{2}(C^2 - C \cdot K_S) - rd.\]
		Each fiber $F$ satisfies $F^2 = 0$ and $g(F) = 0$, hence $F \cdot K_S = -2$ and $C \cdot K_S = -2t$. Moreover, by \cite[V.2.4]{hartshorne2013algebraic} and Riemann-Roch on curves, $\chi(\mathcal{O}_S) = \chi(\mathcal{O}_B) = 1 - g(B)$. Observing that $C^2 = 0$, we obtain (i).

		(ii) The first part is immediate. For the second part, simply note that on each fiber $F$, $(K_S + H) \cdot F = -1$, so that
		\( \mathcal{L}(K_S + H)|_F \cong \mathcal{O}_{\mathbb{P}^1}(a_i - 1),\)
		and, hence,
		\(H^1(\mathcal{L}(K_S + H)) = 0\) if and only if \(a_i \ge 0 \text{ for all } i.\)
	\end{proof}
	
	\begin{proof}[Proof of Corollary \ref{prop c^2=0}] Suppose that $\mc{E}$ is an Ulrich bundle with $c_1(\mc{E})^2=0$ and let $(C,W,\mc{L})$ denote the triple arising from 
		%the construction carried out in the proof of
		Theorem \ref{main}. Then, by Remark \ref{LM}, $\mc{E}\cong \pi^*\mc{G}(\det(\mc{F}))$ where $\mc{G}$ is a rank $r$ vector bundle on $B$ with vanishing cohomology. In the \hyperlink{dimmain}{Proof of Theorem \ref{main}} we saw that the curve $C$ is a divisor in the 
		$\det (\mc{E})$. Set $\mc{G'}=\mc{G}(\det (\mc{F}))$ and, using the fact that pullback commutes with determinant, observe that
		\(\det (\mc{E})=\det (\pi^*\mc{G'})=\pi^*\det (\mc{G'})\). 
		But $\det (\mc{G'})$ is a divisor on the base curve $B$ and therefore corresponds to an effective divisor $D=P_1+ \ldots +P_t$ on $B$. Pulling back, it follows that 
		\(C\in |\pi^*D|=|\pi^*P_1+\ldots + \pi^*P_t|.\)
		
		Since each $\pi^*P_i$ is exactly the fiber $F_i\cong \mathbb{P}^1$ we conclude that $C$ is linearly equivalent to the sum of these fibers
		%\[C \sim F_1+\ldots +F_t.\]
		and, since $C$ is smooth by Theorem \ref{main}, no two fibers $F_i$ can coincide or meet nontrivially. Hence, 
		\(C = \bigsqcup_{i=1}^t F_i\)
		is the disjoint union of $t$ distinct copies of $\mathbb{P}^1$, which, specifically, are fibers of the ruling $\pi.$
		
		In particular, observe that 
		\[t=\deg (c_1 (\mc{G'}))=\deg(c_1(\mc{G}\otimes\det (\mc{F})))= \deg (c_1(\mc{G}))+r\deg (c_1 (\mc{F}))\]
		where in the last equality we used \cite[Prop. 5.18]{eisenbud20163264}.
		
		Since $H^i(\mc{G})=0$ for all $i$, Riemann–Roch on the curve $B$ gives
		\[0=\chi (\mc{G})= \deg (\mc{G})+r(1-g(B))\]
		hence $\deg (\mc{G})=r(g(B)-1)$, that is, $\deg (c_1(\mc{G}))=r(g(B)-1).$
		
		By \cite[Appendix A.3]{hartshorne2013algebraic}, one immediately obtains
		\( d = \pi^* c_1(\mathcal{F}) \cdot H.\)
		Since $\pi^* c_1(\mathcal{F})$ is a sum of $\deg(c_1(\mathcal{F}))$ fibers, each meeting $H$ once, it follows that $d = \deg(c_1(\mathcal{F}))$. Hence,
		\[ t=  \deg (c_1(\mc{G}))+r\deg (c_1 (\mc{F})) = r(g(B)-1+d)\]
		as required.
		
		Next, we show that $\mathcal{L}$ is trivial. Condition (iii) of Theorem \ref{main}, together with the preceding computation of $t$ and Remark \ref{remark tecnico}(i), immediately gives 
		\(\deg(\mathcal{L}) = 0.\)
		Since $C = \bigsqcup_{i=1}^t F_i$ and $H^1(\mathcal{L}(K_S+H)) = 0$ by condition (i) of Theorem \ref{main}, Remark \ref{remark tecnico}(ii) implies that
		\( \mathcal{L} \cong \bigoplus_{i=1}^t \mathcal{O}_{\mathbb{P}^1}(a_i) \)
		with $\sum_i a_i = \deg(\mathcal{L}) = 0$ and $a_i \ge 0$ for each $i$. It follows that $a_i = 0$ for all $i$, and hence $\mathcal{L} \cong \mathcal{O}_C$.
		Finally, condition (ii) of Theorem \ref{main} ensures that the multiplication map
		\(
		W \otimes H^0(\mathcal{O}_S(K_S + 2H)) \to H^0(\mathcal{O}_C(K_S + 2H))
		\)
		is injective. Since $(K_S + 2H) \cdot C = 0$, we have $H^0(\mathcal{O}_C(K_S + 2H)) = H^0(\mathcal{O}_C)$, which concludes the first part of the proof.

		Conversely, 
		%		suppose $C$ is the disjoint union of $t=r(g(B)-1+d)$ fibers of $S$, $\mc{L}=\mc{O}_C$ and $W\subseteq H^0(\mc{L})$ is a base–point–free subspace of dimension $r$ such that the multiplication map $W\otimes H^0(\mc{O}_S(K_S+2H))\to H^0(\mc{O}_C)$ is injective.
		by Remark \ref{remark tecnico}(i) and setting $\mc{L}:=\mc{O}_C$, one has 
		\[r\chi(\mc{O}_{S})+g(C)-1-C\cdot K_S-rH^2=0, \qquad \deg(\mc{L})=r\chi(\mc{O}_{S})+g(C)-1-C\cdot K_S-rH^2.\]
		Remark \ref{remark tecnico}(ii) then gives $H^1(\mc{L}(K_S+H))=0$. Recalling that $K_S=-2H+\pi^*(K_B+ \det \mc{F})$ and that $\deg (\det \mc{F})=d$ by the first part of the proof, we compute
		\begin{center}
			$\frac{r}{2}(K_S+3H)\cdot H= \frac{r}{2}(H^2+\deg(K_B+\det \mc{F}))
			=r(g(B)-1+d)=t=\deg(C).$
		\end{center}
		Therefore, the triple $(C,W,\mc{L})$ satisfies conditions (i), (ii) and (iii) of Theorem \ref{main}, giving rise to an Ulrich bundle $\mc{E}$ on $S$ such that $c_1(\mc{E})^2=C^2=0$ and thus concluding the proof.
	\end{proof}
	
	\begin{oss}\label{caso speciale}
		In the setting of Corollary \ref{prop c^2=0}, suppose $\deg(S) = d \geq 2$. A straightforward computation shows that
		\[t=1 \iff r=1,\ d=2,\  g(B)=0,\]
		which means that $S \cong \mathbb{P}^1 \times \mathbb{P}^1$. By , on this surface every Ulrich vector bundle splits as
		\[\mathcal{O}_{\mathbb{P}^1 \times \mathbb{P}^1}(1,0)^{\oplus \alpha} \oplus \mathcal{O}_{\mathbb{P}^1 \times \mathbb{P}^1}(0,1)^{\oplus \beta}, \quad \alpha, \beta \geq 0.\]
		Since in this case $r = 1$, the only possibilities for $\mathcal{E}$ are $\mc{O}_{\mathbb{P}^1\times \mathbb{P}^1}(1,0)$ or $\mc{O}_{\mathbb{P}^1\times \mathbb{P}^1}(0,1)$. 
		
		Moreover, this special case is covered by the Lopez–Mu\~noz theorem \cite[Thm. 1]{lopez2021classification}, by taking \(\mc{F}=O_{\mathbb{P}^1}(1)^{\oplus 2}, \ 
		H=\mc{O}_{\mathbb{P}\mc{F}}(1), \  \mc{G}=O_{\mathbb{P}^1}(-1).\)
		
	\end{oss}
	\vspace{1ex}
	Now, we can finally turn to the proof of Corollary \ref{connessione}.
	\hypertarget{dimconnessione}{}
	\begin{proof}[Proof of Corollary \ref{connessione}] 
		First suppose that $C$ is irreducible. If $C^2>0$, then by Corollary \ref{LM}, $(S, \mc{O}_S(1), \mc{E})$ cannot be of the form$ (\mathbb{P}\mc{F}, \mc{O}_{\mathbb{P}\mc{F}}(1), \pi^*(\mc{G}(\det (\mc{F})))$, so we are in case (i). On the other hand, if $C^2=0$, by Corollary \ref{LM}, the only possibility for $C$ to remain irreducible is the special situation described in Remark \ref{caso speciale}, that is, (ii). 
		
		Conversely, assuming that $C$ is reducible, we aim to show $c_1(\mathcal{E})^2 = 0$. Once this is established, Corollary \ref{LM} excludes case (i), and case (ii) is precluded as well, since it would imply that $C$ is a line.
		
		Consider the following short exact sequence 
		\[ 0\to \mc{O}_S(-C)\to \mc{O}_S\to \mc{O}_C\to 0\]
		and the induced long exact sequence in cohomology 
		\[0\to H^0(\mc{O}_S(-C))\to H^0( \mc{O}_S)\xrightarrow{\varphi} H^0(\mc{O}_C)\to H^1(\mc{O}_S(-C))\to \ldots  \]
		Since $C$ is smooth and reducible, we have $h^0(\mc{O}_C)\geq 2$, whereas $H^0(\mc{O}_S)\cong \mathbb{C}$. It follows that the map $\varphi$ fails to be surjective, and hence $H^1(\mc{O}_S(-C))\neq 0$. On the other hand, from the proof of Theorem \ref{main} we know that $\det(\mc{E})=C$, so that, by Serre's duality, we have
		\(h^1(\mc{O}_S(K_S+\det(\mc{E})))=h^1(\mc{O}_S(-C))\neq 0. \)
		Moreover, the bundle $\mc{E}$ is globally generated, whence $\det (\mc{E})$ is a globally generated line bundle on $S$; hence, $\det(\mc{E})$ is nef. An application of the Kawamata–Viehweg vanishing theorem \cite[Thm. 4.3.1]{lazarsfeld2017positivity} then yields $c_1(\mc{E})^2=0$, concluding the proof. 
	\end{proof}
	
	\section{Applications to surfaces in \(\mathbb{P}^3\)}
	In the particular case of surfaces in $\mathbb{P}^3$, Theorem \ref{main} together with the bounds on the genus of the curve $C$ simplify to the following form.
	\begin{thm} \label{main P3}
		Let \( S \subset \mathbb{P}^3 \) be a smooth projective surface of degree $d\geq 2$, embedded by the linear system $|H|$, where $H\in |\mc{O}_S(1)|$. Then, there exists an Ulrich bundle $\mc{E}$ of rank $r$ on $S$ if and only if there exists a smooth (possibly disconnected) curve $C \subset S $ of genus $g$ and a line bundle $\mc{L}$ on $C$, with $h^0(\mc{L})=r$, such that: 
		
		\begin{enumerate}[label=(\roman*)]
			\item $H^1(C, \mc{L}((d-3)H))=0$; %where $\mc{L}=\mc{O}_C(D)$ for some divisor D on $\mc{C}$;
			\item the multiplication map 
			\begin{center}
				$\varphi: H^0(\mc{L})\otimes H^0(S,\mc{O}_{S}((d-2)H)) \to H^0(\mc{L}((d-2)H))$
			\end{center}
			is injective;
			\item $ \deg(C)= \frac{r}{2}d(d-1)$ and \\[1ex] $\deg(\mc{L})= \frac{r}{2}(2\binom{d-1}{3}-d(d-2)(d-3)+2)+g-1$.
		\end{enumerate}
	\end{thm}
	
	\begin{oss} \label{disequazioni P3}
		The inequalities from Remark \ref{bounds} concerning the genus $g$ of the curve $C$ of Theorem \ref{main P3} take the following form:
		\begin{enumerate}[label=(\roman*)]
			\item $g\leq \frac{r}{8}d(d-4)((r+2)d+2(r-1))+\frac{9r^2d}{8}+1 $
			\item $g\geq \frac{r(r+1)}{4}d(d-1)(d-4)-r^2(\binom{d-1}{3}-d+1)+1$. 
		\end{enumerate} 
		
	\end{oss}
	
	The question of connectedness becomes particularly simple for surfaces in $\mathbb{P}^3$, as it reduces to a unique case where the curve fails to be connected.
	
	\begin{cor} \label{connessione P^3}
		Let \( S \subset \mathbb{P}^3 \) be a smooth projective surface of degree $d\geq 2$, embedded by the linear system $|H|$, where $H\in |\mc{O}_S(1)|$. Let $\mc{E}$ be the Ulrich bundle on $S$ corresponding to the pair $(C, \mc{L})$, as in Theorem \ref{main P3}. Then $C$ is disconnected if and only if
		\begin{equation} \label{connessione curva P^3}
			\begin{aligned}
				(S, \mc{O}_S(1), \mc{E})\cong (\mathbb{P}^1\times \mathbb{P}^1, \mc{O}_{\mathbb{P}^1\times \mathbb{P}^1}(1), \mc{O}_{\mathbb{P}^1\times \mathbb{P}^1}(1,0)^{\oplus r}) \   \text{or} \\[1ex] (S, \mc{O}_S(1), \mc{E})\cong (\mathbb{P}^1\times \mathbb{P}^1, \mc{O}_{\mathbb{P}^1\times \mathbb{P}^1}(1), \mc{O}_{\mathbb{P}^1\times \mathbb{P}^1}(0,1)^{\oplus r})
			\end{aligned}
		\end{equation}
		with $r \geq 2$. 
		
		In particular, in this case $C$ is the disjoint union of exactly $r$ lines. 
		
	\end{cor}
	\begin{proof}
		If (\ref{connessione curva P^3}) holds, then both (i) and (ii) of Corollary \ref{connessione} fail. Indeed, (i) fails choosing $\mc{F}=O_{\mathbb{P}^1}(1)^{\oplus 2}, \ B=\mathbb{P}^1, \  H=\mc{O}_{\mathbb{P}\mc{F}}(1),  \ \mc{G}=O_{\mathbb{P}^1}(-1)$
		and (ii) fails since $r\geq 2$. 
		
		Hence, $C$ is disconnected and, in particular, by Lemma \ref{prop c^2=0}, it is a disjoint union of $t=r(g(B)-1+d)= r$ lines. 
		
		On the other hand, by Corollary \ref{connessione}, if the curve $C$ is disconnected, $S$ is a ruled surface 
		\[(S, \mc{O}_S(1), \mc{E})\cong (\mathbb{P}\mc{F}, \mc{O}_{\mathbb{P}\mc{F}}(1), \pi^*(\mc{G}(\det (\mc{F}))),\]
		where the data $(\mc{F}, B,\mc{G})$ are exactly those specified in Corollary \ref{connessione}(i), and, in particular, the triple $(S, \mc{O}_S(1), \mc{E})$ is not isomorphic to either of the spinorial configurations of Corollary \ref{connessione}(ii). 
		%			\[(\mathbb{P}^1\times \mathbb{P}^1, \mc{O}_{\mathbb{P}^1\times \mathbb{P}^1}(1), \mc{O}_{\mathbb{P}^1\times \mathbb{P}^1}(1,0)), \quad (\mathbb{P}^1\times \mathbb{P}^1, \mc{O}_{\mathbb{P}^1\times \mathbb{P}^1}(1), \mc{O}_{\mathbb{P}^1\times \mathbb{P}^1}(0,1)).\]
		Comparing the two canonical expressions $K_S=(d-4)H$ and $K_S=-2H+\pi^*(K_B+ \det (\mc{F})),$
		%			\begin{center}
			%				$K_S=(d-4)H$ and $K_S=-2H+\pi^*(K_B+ \det (\mc{F})),$
			%			\end{center}
		one has
		\[(d-2)H=\deg (K_B+ \det (\mc{F}))f,\]
		where $f$ is a generic fiber of $S$.
		Intersecting with a fiber $f$,
		%			\[(d-2)H\cdot f=\deg (K_B+ \det (\mc{F}))f^2\]
		%			and, recalling that $H\cdot f=1$ and $f^2=0$
		one has $d=2$.
		Hence, \newline $\deg (K_B+ \det (\mc{F}))=0$, that is, $2g(B)-2+d=0$, which implies $g(B)=0$.
		
		Thus, $B=\mathbb{P}^1$ and $S\cong \mathbb{P}^1\times \mathbb{P}^1$.
		Since $H^{q}(B,\mc{G})=0$ for all $q\geq 0$ and every vector bundle on $\mathbb{P}^1$ splits as a direct sum of line bundles, we have that $\mc{G}\cong \bigoplus_{i=1}^{r}\mc{O}_{\mathbb{P}^{1}}(-1)$.
		Moreover $\deg(\det\mc{F})=-\deg (K_B)=2$, hence $\det\mc{F}\cong\mc{O}_{\mathbb{P}^{1}}(2)$.
		Hence,
		\(
		\mc{G}\otimes\det\mc{F}\cong
		\mc{O}_{\mathbb{P}^{1}}(1)^{\oplus r},
		\)
		so that $\mc{E}$ coincides with the pull-back $\pi^{*}(\mc{O}_{\mathbb{P}^{1}}(1)^{\oplus r})$, that is, 
		\begin{center}
			$\mc{E}\cong \mc{O}_{\mathbb{P}^1\times \mathbb{P}^1}(1,0)^{\oplus r}$ or $\mc{E}\cong \mc{O}_{\mathbb{P}^1\times \mathbb{P}^1}(0,1)^{\oplus r}$.
		\end{center} 
		and, since the single-spinor cases have been excluded, we must have $r \geq 2$.
		
		This exactly gives the configuration in (\ref{connessione curva P^3}), completing the proof.
	\end{proof}
	\section{Ulrich bundles and Noether-Lefschetz Loci} The goal of this section is to show that \textit{smooth surfaces of fixed degree $d \geq 4$ in $\mathbb{P}^3$ admitting an Ulrich line bundle form a component of the Noether--Lefschetz locus}. Before turning to the precise statement and its proof, we provide some preliminary remarks and background. 
	\begin{notation}
		The parameter space of curves of degree $n$ and genus $g$ in $\mathbb{P}^3$ is denoted $H_{n,g}$ and it is called the \emph{Hilbert scheme} of curves of degree $n$ and genus $g$. One can verify that $H_{n,g}$ carries the natural structure of a projective variety (see \cite[Remark 9.8.1 and Ch. IV, \S6]{hartshorne2013algebraic}).
		
		For a fixed integer $d\geq 1$, we denote by $\mathbb{P}^N=\mathbb{P}^{\binom{d+3}{3}-1}$ the projective space whose points correspond to surfaces of degree $d$ in $\mathbb{P}^3$ and by $\mc{U}_d\subseteq \mathbb{P}^N$ the open subset consisting of points corresponding to smooth surfaces.
	\end{notation}	
	
	\begin{oss}[\textbf{Noether-Lefschetz}]
		According to the classical Noether-Lefschetz Theorem (see for instance \cite{lefschetz1924analysis, grothendieckcohomologie, griffiths1985noether}), a very general surface $S \subset \mathbb{P}^3$ of degree $d \geq 4$ has Picard group generated by the hyperplane class, that is $\mathrm{Pic}(S) \cong \mathbb{Z}H$. The exceptional locus of surfaces for which the Picard group is larger than $\mathbb{Z}H$ is the so-called \textit{Noether--Lefschetz locus} $NL(d)$. Recall that a property holds for a very general point of a projective variety \(X\) when it is satisfied outside a countable union of proper closed subvarieties of \(X\); then $NL(d)$ can be described as a countable union of proper closed subvarieties of the parameter space $\mathcal{U}_d$ of smooth degree-$d$ surfaces. If $W$ is an irreducible component of $NL(d)$, then its codimension in $\mathcal{U}_d$ is bounded between $d-3$ and the geometric genus $p_g(d) = \binom{d-1}{3}$ (see \cite{carlson1983infinitesimal} or \cite{green1984koszul}). Those components which achieve the maximal codimension $p_g(d)$ are called \emph{general components} of the Noether--Lefschetz locus.
	\end{oss}	
	
	We next present a corollary of Theorem \ref{main} that provides a characterization of Ulrich line bundles on a surface $S$, formulated in the spirit of their correspondence with curves on the surface. Although the statement essentially coincides with \cite[Prop. 2.4]{ciliberto2023ulrich} (for the case $S\subset \mathbb{P}^3$, see also \cite[Prop. 6.2]{beauville2000determinantal} together with \cite[Prop. 2.2]{beauville2018introduction}), we present a short outline of an independent argument relying on Theorem \ref{main}.
	\begin{cor}[\textbf{Ulrich line bundles on surfaces}]\label{main rango 1}
		Let \( S \subset \mathbb{P}^N \) be a smooth projective surface of degree $d\geq 2$, embedded by the linear system $|H|$, where $H\in |\mc{O}_S(1)|$. Then, there exists an Ulrich line bundle $\mc{E}$ on $S$ if and only if there exists a smooth (possibly disconnected) curve $C \subset S $ of genus $g$ such that: 
		\begin{enumerate} [label=(\roman*)]
			\item $H^1(\mc{O}_C(K_S+H))=0$;
			\item $H^0(\mc{O}_S(K_S+2H-C))=0$
			\item $\deg(C)= \frac{1}{2}(K_S+3H)\cdot H, \quad g=C\cdot K_S+1+d-\chi(\mc{O}_S)$. 
		\end{enumerate}
	\end{cor}
	\begin{proof}
		Write the multiplication map supplied by Theorem \ref{main} as $\varphi = \rho \circ \psi$
		with 
		\newline \(
		\psi : W \otimes H^0(\mathcal{O}_S(K_S+2H)) \rightarrow 
		W \otimes H^0(\mathcal{O}_C(K_S+2H)),
		\)
		and where $\rho$ is multiplication by the generator $\sigma \in W$. Since $W$ is base-point-free of dimension one, $\sigma$ restricts to nonzero constants on each connected component of $C$, so $\rho$ is injective; thus $\varphi$ is injective if and only if $\psi$ is.
		
		If $\mathcal{E}$ is an Ulrich line bundle, Theorem \ref{main} furnishes a triple $(C,W,\mathcal{L})$ with 
		$\mathcal{L} \cong \mathcal{O}_C$, and hence conditions (i)--(iii) hold. Conversely, given (i)--(iii), 
		the factorization and injectivity just proved verify condition (ii) of Theorem \ref{main}, and the numerical 
		conditions yield the required degrees, so Theorem \ref{main} produces an Ulrich line bundle on $S$.
		\end{proof}		
		For surfaces in $\mathbb{P}^3$, a curve satisfying conditions (i)--(iii) of Corollary \ref{main rango 1}  is necessarily arithmetically Cohen–Macaulay. Before proving this statement, we recall the definition of ACM curves.
	\begin{definition}
		A smooth projective curve $C\subseteq \mathbb{P}^N$ is said \emph{arithmetically Cohen-Macaulay (ACM)} if it satisfies
		\(H^1(\mc{I}_{C/\mathbb{P}^N}(j))=0 \text{ for all } j\in \mathbb{Z}.\)
	\end{definition}
	\begin{oss}\label{ACM equivale risoluzione}
		A smooth curve $C\subset \mathbb{P}^3$ is ACM if and only if it admits a minimal free resolution 
		\begin{equation} \label{risoluzione libera min}
			0\to \bigoplus_{i=1}^{n+1}\mc{O}_{\mathbb{P}^3}(-m_i)\to \bigoplus_{j=1}^{n+2}\mc{O}_{\mathbb{P}^3}(-d_j)\to \mc{I}_{C/\mathbb{P}^3}\to 0. 
		\end{equation}
		Indeed, if $C$ is ACM, the existence of (\ref{risoluzione libera min}) is proved in \cite[Thm. A]{beauville2000determinantal}.
		Conversely, it suffices to observe that \(H^1(\bigoplus_{j=1}^{n+2}\mc{O}_{\mathbb{P}^3}(-d_j+k))= H^2(\bigoplus_{i=1}^{n+1}\mc{O}_{\mathbb{P}^3}(-m_i+k))=0\) for all integres $k$. 
		Moreover, the integers $m_i$ satisfy the strict inequalities 
		\begin{equation}\label{risoluzione libera diseq}
			m_i > \min\{d_j : j=1,\dots,n+2\}.
		\end{equation} 
		Indeed, writing the first map in the resolution as a matrix $[A_{ij}]$, minimality of the resolution forces each entry to have degree $m_i - d_j$ if positive, and zero otherwise. If for some $i$ one had $m_i \le \min\{d_j\}$, all entries in the $i$-th row would vanish. But the ideal $I_{C/\mathbb{P}^3}$ is generated by the $(n+1)\times(n+1)$ minors of this matrix, so a zero row would make all these minors vanish, so that $I_{C/\mathbb{P}^3}=0$, a contradiction. Hence the inequalities hold for every $i$.
		
	\end{oss}
	\begin{oss}\label{curve ulrich}
		Let \( S \subset \mathbb{P}^3 \) be a smooth surface of degree $d\geq 4$, embedded by the linear system $|H|$, where $H\in |\mc{O}_S(1)|$, and let $C\subset S$ be a smooth curve of genus $g$ as in Corollary \ref{main rango 1}, that is, 
		
		\begin{enumerate} [label=(\roman*)]
			\item $H^1(\mc{O}_C((d-3)H))=0$;
			\item $H^0(\mc{O}_S((d-2)H-C))=0$
			\item $\deg(C)= \frac{1}{2}d(d-1), \quad g=\frac{1}{6}(d-2)(d-3)(2d+1)$. 
		\end{enumerate}
		
		Then:
		\begin{enumerate}[label=(\roman*)]
			\item $C$ is irreducible;
			\item $\mc{O}_S(C)$ is an Ulrich line bundle on $S$;
			\item $C$ is ACM.
		\end{enumerate}
	\end{oss}
		\begin{proof}
			Item (i) follows directly from Corollary \ref{connessione P^3}, simply by noting that $d\geq 4$, while (ii) is an immediate consequence of the constructions of Theorem \ref{main} and Corollary \ref{main rango 1}.
			
			To prove (iii), fix an integer $j\in \mathbb{Z}$ and consider the short exact sequence
			\begin{equation}\label{sec ideali}
				0\to \mc{I}_{S/\mathbb{P}^3}(j)\to \mc{I}_{C/\mathbb{P}^3}(j)\to \mc{I}_{C/S}(j)\to 0.
			\end{equation}
			Since $H^1(\mc{O}_{\mathbb{P}^3}(\ell)) = 0$ for all $\ell \in \mathbb{Z}$, we immediately obtain $H^1(\mc{I}_{S/\mathbb{P}^3}(j)) = 0$.
			
			Moreover, by \cite[(3.1)]{beauville2018introduction} together with point (ii), we have $H^1(\mc{O}_S(C)(\ell)) = 0$ for all $\ell \in \mathbb{Z}$, equivalently $H^1(\mc{O}_S(aH - C)) = 0$ for all integers $a$. Hence, $H^1(\mc{I}_{C/S}(j)) = 0$.
			
			Then, by the exactness of (\ref{sec ideali}), $H^1(\mc{I}_{C/\mathbb{P}^3}(j)) = 0$ for all $j$, proving that $C$ is ACM.
		\end{proof}
		To prove our main theorem, we will follow the approach of Ciliberto-Lopez \cite{ciliberto1991existence}, whose Lemma 1.2 provides a criterion for identifying components of the Noether--Lefschetz locus. The following two lemmas establish the preliminary results required to apply this criterion to our curves.

	\begin{lem}\label{prima condizione}
		Let $C\subset \mathbb{P}^3$ be a smooth $ACM$ curve and let $d\geq 4$ be an integer. Then \(\mc{I}_C \text{ is $(d-1)$-regular if and only if } H^1(\mc{O}_C(d-3))=0.\)
	\end{lem}
	\begin{proof}
		Since $C$ is ACM, $\mathcal{I}_C$ is $(d-1)$-regular if and only if
		\(
		H^2(\mathcal{I}_C(d-3)) = H^3(\mathcal{I}_C(d-4)) = 0.
		\)
		From the short exact sequence
		\[
		0 \to \mathcal{I}_C \to \mathcal{O}_{\mathbb{P}^3} \to \mathcal{O}_C \to 0,
		\]
		twisting by $\mathcal{O}_{\mathbb{P}^3}(d-3)$ and passing to cohomology gives $H^2(\mathcal{I}_C(d-3)) \cong H^1(\mathcal{O}_C(d-3))$, while $H^3(\mathcal{I}_C(d-4)) = 0$ automatically, implying the lemma.
	\end{proof}
%	\begin{proof}
%		$\mc{I}_C $ is $(d-1)$-regular if and only if $H^i(\mc{I}_C(d-1-i))=0$ for all $i>0$, that is, since $C$ is assumed to be $ACM$,  
%		\(H^2(\mc{I}_C(d-3))=0, \text{ and } H^3(\mc{I}_C(d-4))=0.\) Now, consider the short exact sequence 
%		\[0\to \mc{I}_C\to\mc{O}_{\mathbb{P}^3}\to \mc{O}_C\to 0.\]
%		Twisting by $\mc{O}_{\mathbb{P}^3}(d-3)$ 
%		%and $\mc{O}_{\mathbb{P}^3}(d-4)$ 
%		and passing to cohomology, one has
%		\[H^2(\mc{I}_C(d-3))=0 \iff H^1(\mc{O}_C(d-3))=0,\]
%		while, observing that $H^2(\mc{O}_C(d-4))=0 \text{ and } H^3(\mc{O}_{\mathbb{P}^3}(d-4))=0, $ the condition $H^3(\mc{I}_C(d-4))=0$ is automatically satisfied.
%		
%		It follows that 
%		\(h^2(\mc{I}_C(d-3))=h^3(\mc{I}_C(d-4))=0  \text{ if and only if }h^1(\mc{O}_C(d-3))=0,\)
%		that is, the statement.
%	\end{proof}
	
	\begin{lem}\label{seconda condizione}
		Let $C\subset \mathbb{P}^3$ be a smooth irreducible ACM curve such that $\mc{I}_C$ is $(d-1)$-regular, with $d\geq 4$. Then
		\begin{enumerate}[label=(\roman*)]
			\item there exists a smooth surface $S$ of degree $d$ containing $C$;
			\item for any smooth surface $S$ lying in the linear system $|\mc{I}_C(d)|$, one has the isomorphism
			\(H^0(\mc{I}_C((d-2))\cong H^0(\mc{O}_S((d-2)H-C))\),
			where $H\in |\mc{O}_S(1)|.$
		\end{enumerate}
		
	\end{lem}
	
	\begin{proof}
		(i) Observe that $\mc{I}_C(d-1)$ is 0-regular, hence, it is globally generated. Then, by \cite[Thm. 1.8.5(ii)]{lazarsfeld2017positivity} the natural map
		\(H^0(\mc{I}_C(d-1))\otimes H^0(\mc{O}_{\mathbb{P}^3}(1))\to H^0(\mc{I}_C(d))\)
		is surjective and, since $H^0(\mc{I}_C(d-1))\neq 0$, one has $H^0(\mc{I}_C(d))\neq 0$. Hence, there exists a surface $S\in |\mc{I}_C(d)|$, that is, a degree $d$ surface containing $C$. Actually, one can prove the existence of a smooth degree $d$ surface containing $C$ -- see for istance \cite[Proof of Lemma 4.1]{bruzzo2021existence}. 
		
		(ii) Let $S\in |\mc{I}_C(d)|$ be a smooth surface. Then, one has the following short exact sequence
		\begin{equation} \label{sec d-2}
			0\to \mc{I}_{S/\mathbb{P}^3}(d-2)\to \mc{I}_{C/\mathbb{P}^3}(d-2)\to \mc{I}_{C/S}(d-2)\to 0
		\end{equation}
		Observe that $\mc{I}_{S/\mathbb{P}^3}(d-2)\cong \mc{O}_{\mathbb{P}^3}(-2)$ and, hence, 
		\(H^0(\mc{I}_{S/\mathbb{P}^3}(d-2))=H^1(\mc{I}_{S/\mathbb{P}^3}(d-2))=0.\)
		Then, from the cohomology of the exact sequence (\ref{sec d-2}), one has
		\[H^0(\mc{I}_C(d-2))\cong H^0(\mc{I}_{C/S}(d-2))\cong H^0(\mc{O}_S((d-2)H-C)),\]
		that is, (ii). 
	\end{proof}
	We are now ready to state and prove the main result of this section.
	\begin{definition}
		Let $d$ be a positive integer. We define 
		\[Ul_{r,d}:=\{S\in \mc{U}_d \text{ : $uc(S)=r$}\}.\]
	\end{definition}
	\begin{oss}
		The locus $\mathrm{Ul}_{r,d}$ is constructible. Indeed, for any $n \ge 2$ and integers $d,r \ge 1$, let $V_{r,d}$ denote the set of degree-$d$ hypersurfaces in $\mathbb{P}^n$ carrying at least one Ulrich bundle of rank $r$. Writing $S_k$ for homogeneous polynomials of degree $k$, consider the determinant map 
		$\alpha : \operatorname{Mat}_{r \times r}(S_1) \to S_{rd}$ 
		and the map 
		$\beta : S_d \to S_{rd}, F \mapsto F^r$. 
		Both maps are regular, so by Chevalley's theorem $\beta^{-1}(\operatorname{Im}\alpha) = V_{r,d}$ is constructible. The subset where the minimal Ulrich rank equals $r$, 
		$\mathrm{Ul}_{r,d} = V_{r,d} \setminus (V_{1,d} \cup \cdots \cup V_{r-1,d})$, 
		is then constructible as well. In this paper we focus on $\mathbb{P}^3$, but the argument holds in any dimension.
	\end{oss}
	\begin{thm}\label{componente NL}
		Let $d\geq 4$ be an integer. Then, $Ul_{1,d}$ is the general component of the Noether-Lefschetz locus $NL(d)$ made of surfaces containing a smooth ACM curve whose ideal is given by the $(d-1)\times (d-1)$ minors of a $(d-1)\times d$ matrix of linear forms. 
	\end{thm}

	\begin{proof}
		
		Let \(C \subset \mathbb{P}^{3}\) be a curve as in the statement.  
		It is well known (see \cite[Thm. 6.2]{peskine1974liaison}) that \(C\) admits a minimal free resolution
		\begin{equation}\label{risoluzione top}
			0\to \mc{O}_{\mathbb{P}^3}(-d)^{\oplus (d-1)}\to \mc{O}_{\mathbb{P}^3}(1-d)^{\oplus d}\to \mc{I}_C\to 0.
		\end{equation}
		Observe that such a curve exists from \cite[Thm. 6.2]{peskine1974liaison}. 
		Then, $C$ is ACM by Remark \ref{ACM equivale risoluzione}. 
		In addition, $\mc{I}_C$ is $(d-1)$-regular. Indeed, from the long exact sequence in cohomology induced by (\ref{risoluzione top}) twisted by $\mc{O}_{\mathbb{P}^3}(d-1-i)$, we have the segment
		\[\ldots \to H^i(\mc{O}_{\mathbb{P}^3}(-i))^{\oplus d}\to H^i(\mc{I}_C(d-1-i))\to H^{i+1}(\mc{O}_{\mathbb{P}^3}(-1-i))^{\oplus (d-1)}\to \ldots \] 
		Hence, it follows that $H^i(\mc{I}_C(d-1-i))=0$ for all $i>0$, simply by observing that
		\[H^{i+1}(\mc{O}_{\mathbb{P}^3}(-1-i))=H^i(\mc{O}_{\mathbb{P}^3}(-i))=0 \text{ for $i=1,2,3$ }.\]
		In addition, by \cite[Prop. 3.1]{peskine1974liaison} we have that
		\[d':=\deg(C)=\frac{1}{2}d(d-1);\]
		\[ g:=g(C)= \frac{1}{6}(d-2)(d-3)(2d+1).\]
		We define
		\begin{equation}\label{W}
			W:=\overline{\{[C]\in H_{d',g} \  : \  \text{there exists minimal free resolution as in (\ref{risoluzione top})}\}}.
		\end{equation}
		Then, $W$ is a component of $H_{n,g}$ and a general point $[C]$ of $W$ is a curve that admits the minimal free resolution (\ref{risoluzione top}) (see \cite{Ellingsrud1975}), hence, it is smooth and irreducible, ACM and such that $\mc{I}_C$ is $(d-1)$-regular. Therefore, by \cite[Lemma 1.2]{ciliberto1991existence}, $W(d)=Im(\pi_1)$ is a component of $NL(d)$. 
		Moreover, by Lemma \ref{prima condizione}, $H^1(\mc{O}_C(d-3))=0$. 
		Note  that $H^0(\mc{O}_{\mathbb{P}^3}(-1))=H^1(\mc{O}_{\mathbb{P}^3}(-2))=0$, hence, $H^0(\mc{I}_C(d-2))=0$.
		By Lemma \ref{seconda condizione}, it follows that for any smooth surface $S\in |\mc{I}_C(d)|$ we have $H^0(\mc{O}_S((d-2)H-C))=0$, where $H\in |\mc{O}_S(1)|$. Hence, any $C\in W$ satisfies point (i), (ii) and (iii) of Remark \ref{curve ulrich}, implying that, by Corollary \ref{main rango 1}, for any smooth surface $S\in|\mc{I}_C(d)|$, we have that $\mc{O}_S(C)$ is an Ulrich line bundle on $S$. 
		
		On the other hand, let $S\subset \mathbb{P}^3$ be a smooth surface of degree $d$. Suppose that there exists a smooth irreducible curve $C\subset S$ satisfying condition (i), (ii) and (iii) of Remark \ref{curve ulrich}. Then, by Remark \ref{curve ulrich}, $C$ is $ACM$ and, by Lemma \ref{prima condizione}, $\mc{I}_C$ is $(d-1)$-regular. Moreover, 	by Lemma \ref{ACM equivale risoluzione}, $C$ admits a minimal free resolution 
		\[0\to \bigoplus_{i=1}^{n+1}\mc{O}_{\mathbb{P}^3}(-m_i)\to \bigoplus_{j=1}^{n+2}\mc{O}_{\mathbb{P}^3}(-d_j)\to \mc{I}_{C/\mathbb{P}^3}\to 0.\] 
		Observe that $H^0(\mc{I}_C(d-2))=0$ by Lemma \ref{seconda condizione} together with condition (ii) of Remark \ref{curve ulrich}. Hence, $d_j=d-1$ for all $j=1, \ldots, n+2$. 
		
		Moreover, $n+2=h^0(\mc{I}_C(d-1))$ which we now prove to be equal to $d$. 
		From the short exact sequence
		\begin{equation}\label{d-1 sequenza}
			0\to \mc{I}_{S/\mathbb{P}^3}(d-1)\to \mc{I}_{C/\mathbb{P}^3}(d-1)\to \mc{I}_{C/S}(d-1)\to 0
		\end{equation}
		we have that $h^0(\mc{I}_{C/\mathbb{P}^3}(d-1))=h^0(\mc{I}_{C/S}(d-1))= h^0(\mc{O}_S((d-1)H-C))=h^2(O_S(C)(-3))$. 
		Note that, by Remark \ref{curve ulrich}, $\mc{O}_S(C)$ is an Ulrich line bundle on $S$. Hence, $h^0(\mc{O}_S(C)(-3))\subseteq h^0(\mc{O}_S(C)(-2))=0$ and \cite[(3.1)]{beauville2018introduction} implies that $h^1(\mc{O}_S(C)(-3))=0$. In particular, 
		\[h^0(\mc{I}_{C/\mathbb{P}^3}(d-1))=h^2(O_S(C)(-3))=\chi (\mc{O}_S(C)(-3))=d,\]
		where the last equality follows Remark \ref{proprietà Ulrich}. 
		Therefore, by (\ref{risoluzione libera diseq}) we have $m_i\geq d$ for all $i$; then, \cite[\S3, (i)]{peskine1974liaison} gives \(d(d-1)\leq \sum_{i=1}^{n+1}m_i=d(d-1),\)
		which implies that $m_i=d$ for all $i$. 
		Hence, $C$ admits a minimal free resolution as is (\ref{risoluzione top}).
		This exactly means that 
		\[S\in W(d)\iff S \text{ admits an Ulrich line bundle},\]
		that is, $Ul_{1,d}=W(d)$, thereby implying that $Ul_{1,d}$ is a component of $NL(d)$. 
		
		We now prove that this component is \emph{general}.
		
		Observe that the curves constructed in $W$, satisfy the vanishing condition \( H^1(N_C) = 0 \), where \( N_C \) denotes the normal bundle of \( C \subset \mathbb{P}^3 \). The vanishing \( H^1(N_C) = 0 \) can be checked explicitly using a version of Kleppe's Lemma (see \cite[Lemma 4.2]{bruzzo2021existence}), from which one has in particular
		\[H^1(N_C)\cong Ext_{\mc{O}_{\mathbb{P}^3}}^2(\mc{I}_C, \mc{I}_C).\]
		Applying the functor Hom($-, \mc{I}_C$) to the exact sequence (\ref{risoluzione top}), we obtain the exact segment
		\[\ldots \to Ext^1(\mc{O}_{\mathbb{P}^3}(-d), \mc{I}_C)^{\oplus (d-1)}\to Ext^2(\mc{I}_C, \mc{I}_C)\to Ext^2(\mc{O}_{\mathbb{P}^3}(1-d), \mc{I}_C)^{\oplus d}\to \ldots \]
		Observe that
		\[Ext^1(\mc{O}_{\mathbb{P}^3}(-d), \mc{I}_C)=Ext^1(\mc{O}_{\mathbb{P}^3}, \mc{I}_C(d))=H^1(\mc{I}_C(d))=0 \quad \text{and}\]
		\[Ext^2(\mc{O}_{\mathbb{P}^3}(1-d), \mc{I}_C)=Ext^2(\mc{O}_{\mathbb{P}^3}, \mc{I}_C(d-1))=H^2(\mc{I}_C(d-1))=0\]
		where the vanishing $H^1(\mc{I}_C(d))=0$ is given by the ACM property and $H^2(\mc{I}_C(d-1))=0$ is a consequence of \cite[Thm. 1.8.5]{lazarsfeld2017positivity}. 
		
		Hence, $Ext^2(\mc{I}_C, \mc{I}_C)=0$, that is, $H^1(N_C)=0$. 
		
		% Observe that $C$ admits a minimal free resolution as in (\ref{risoluzione top}), so that its projective dimension is at most 1. Hence, by \cite[Proposition III.6.10A]{hartshorne2013algebraic}, we have that
		% \[Ext_{\mc{O}_{\mathbb{P}^3}}^i(\mc{I}_C, \mc{F})=0 \text{ for all $i>1$ and for all coherent sheaf $\mc{F}$},\]
		% which naturally gives $Ext_{\mc{O}_{\mathbb{P}^3}}^2(\mc{I}_C, \mc{I}_C)=0$, that is, \( H^1(N_C) = 0 \). 
		
		%	This implies that the component \( W \) is smooth and its dimension is given by
		%	\[
		%	\dim W = \chi(N_C)=h^0(N_C) = 4\deg(C).
		%	\]
		This vanishing immediately yields the identity 
		\(\chi(N_C) = h^0(N_C)\). Now, according to \cite[Chapter 2.a]{HarrisEisenbud1982} or \cite[\S 1.E]{harris2006moduli} one has the inequality
		\[
		\chi(N_C) \leq \dim W \leq h^0(N_C).
		\]
		Since both bounds coincide in this case, it follows that
		\(
		\dim W = \chi(N_C) = 4\deg(C).
		\)
		By applying \cite[Lemma 1.2]{ciliberto1991existence}, one concludes that the associated component \( W(d) \subset \mathrm{NL}(d) \) in the Noether--Lefschetz locus is general.
		This follows from the identity
		\[
		h^0(\mc{O}_C(d-4)) = h^0(\mc{O}_{\mathbb{P}^3}(d-4)) = p_g(d).
		\]
	\end{proof}
	
	\hypertarget{refining}{}
	\section{Refining the lower genus bound for the curve $C$ in Theorem \ref{main} on quartic surfaces in $\mathbb{P}^3$}
	
	Here, we focus on refining the lower bound for the genus of the curve $C$ appearing in Theorem \ref{main P3}, in the context of smooth quartic surfaces in $\mathbb{P}^3$. It will be shown that, for curves arising from Ulrich bundles of minimal rank, this bound improves substantially upon the one given by Bogomolov's inequality in Remark \ref{bounds} or Remark \ref{disequazioni P3}. We precede the statement of the main result with a few simple remarks.
	\begin{oss}
		Let $S\subset \mathbb{P}^3$ be a smooth quartic surface. Then, from Remark \ref{disequazioni P3}, we have the following bounds for the genus $g$ of the curve $C$ of Theorem \ref{main P3}: 
		\[2r^2+1\leq g\leq \frac{9}{2}r^2+1.\] 
	\end{oss}
	\begin{oss}\label{discorso chi}
		Let \(S \subset \mathbb{P}^{N}\) be a smooth surface of degree $d$ and fix an Ulrich bundle
		\(\mc{E}\) of minimal rank \(r\) on \(S\).
		From Remark \ref{stability}, the existence of a non-scalar endomorphism
		of \(\mc{E}\) -- equivalently
		\(h^0(S,\mc{E}\otimes\mc{E}^{*})\geq 2\) -- would force
		\(S\notin Ul_{r,d}\).
		Computing \(h^{0}(S,\mc{E}\otimes\mc{E}^{*})\) directly is,
		however, seldom practical.
		It is easier to work with the Euler characteristic
		\[
		\chi(\mc{E}\otimes\mc{E}^{*})= h^{0}(\mc{E}\otimes\mc{E}^{*})-h^{1}(\mc{E}\otimes\mc{E}^{*})+h^{2}(\mc{E}\otimes\mc{E}^{*}).
		\]
		Assume moreover that \(S\) is a \(K3\) surface. By Serre duality
		\(h^{2}(S,\mc{E}\otimes\mc{E}^{*})
		= h^{0}(S,\mc{E}\otimes\mc{E}^{*})\);
		hence
		\(2h^{0}(S,\mc{E}\otimes\mc{E}^{*})=\chi(\mc{E}\otimes\mc{E}^{*})+ h^{1}(S,\mc{E}\otimes\mc{E}^{*})\ge \chi(\mc{E}\otimes\mc{E}^{*}),\)
		so that, 
		\[h^{0}(S,\mc{E}\otimes\mc{E}^{*})\ge \frac{\chi(\mc{E}\otimes\mc{E}^{*})}{2}.\]
		In particular, if one can prove that
		\(\chi(\mc{E}\otimes\mc{E}^{*})\ge 4\),
		then \(h^{0}(S,\mc{E}\otimes\mc{E}^{*})\ge 2\), contradicting the simplicity established in Remark \ref{stability} for any minimal rank Ulrich bundle.
		This observation justifies the computations in the following remark.
	\end{oss}

	\begin{oss}\label{ExE*}
		Let $S\subset \mathbb{P}^N$ be a smooth projective surface and let $\mc{E}$ be a rank $r$ vector bundle on $S$. Then
		\begin{equation}\label{refining remark}
			\chi(\mc{E}\otimes \mc{E}^*)= r^2\chi(\mc{O}_S)-2rc_2(\mc{E})+(r-1)c_1(\mc{E})^2.
		\end{equation}
		
	\end{oss}
	\begin{proof}
		From \cite[Prop. 5.18]{eisenbud20163264} it follows that 
		\begin{equation}\label{i}
			c_1(\mc{E}\otimes \mc{E}^*)=0.
		\end{equation}
		Observe that, from \cite[Appendix A, \S4]{hartshorne2013algebraic}, we have
		\[[\text{ch}(\mc{E}\otimes \mc{E}^*)]_2= \frac{1}{2}(c_1(\mc{E}\otimes \mc{E}^*)^2-2c_2(\mc{E}\otimes \mc{E}^*)),\]
		that is, by (\ref{i}), $[\text{ch}(\mc{E}\otimes \mc{E}^*)]_2= -c_2(\mc{E}\otimes \mc{E}^*).$
		On the other hand, by a straightforward computation, \cite[\S5.5.2]{eisenbud20163264} gives 
		\[[\text{ch}(\mc{E}\otimes \mc{E}^*)]_2=[\text{ch}(\mc{E})\text{ch}(\mc{E}^*)]_2=-2rc_2(\mc{E}) + (r-1)c_1(\mc{E})^2,\]
		which directly implies that
		\begin{equation}\label{ii}
			c_2(\mc{E}\otimes \mc{E}^*)= 2rc_2(\mc{E}) - (r-1)c_1(\mc{E})^2. 
		\end{equation}
		Then (\ref{refining remark}) is a direct consequence of the Riemann-Roch theorem for vector bundles on smooth projective surfaces, together with (\ref{i}) and (\ref{ii}).
	\end{proof}
	We are now ready to state and to prove the most refined form of the lower bound previously announced.
	\begin{lem}
		Let \(S \subset \mathbb{P}^{3}\) be a smooth quartic surface of degree $d$ and let \(\mathcal{E}\) be a minimal Ulrich bundle of rank \(r\) on \(S\).  Denote by \((C,\mathcal{L})\) the pair associated with \(\mathcal{E}\) via Theorem \ref{main P3}.  Then, the genus $g$ of \(C\) satisfies
		\[
		g\geq 3r^{2}.
		\] 
	\end{lem}
	
	\begin{proof}
		%			Observe that, by Remark \ref{ExE*}, we have
		%			\[\chi(\mc{E}\otimes \mc{E}^*)= 2r^2-2rc_2(\mc{E})+(r-1)c_1(\mc{E})^2\]
		For the Ulrich bundle $\mathcal{E}$ one has
		$c_{1}(\mathcal{E}) = C$ with $C^{2} = 2g-2$, and
		$c_{2}(\mathcal{E}) = g - 1 - 2r$; substituting these expressions in (\ref{refining remark}) gives
		\[\chi(\mc{E}\otimes \mc{E}^*)=6r^2-2g+2.\]
		Assume, for contradiction, that $g \le 3r^{2} - 1$.  Then \[\chi(\mathcal{E}\otimes\mathcal{E}^{*}) \ge 4\]
		which, by Remark \ref{discorso chi}, contradict the simplicity of the minimal-rank Ulrich
		bundle $\mc{E}$.  Consequently, we must have \(g\geq 3r^2\). 
	\end{proof}

	\addcontentsline{toc}{section}{Bibliography}
	
	\bibliographystyle{alphanum}
	\bibliography{Refs}{}
\end{document}